\documentclass[12pt,a4paper]{article}

\pdfoutput=1
\usepackage{graphicx}
\graphicspath{{./Figures/}}
\usepackage{subfigure}

\usepackage{amssymb,latexsym}

\usepackage[nocompress]{cite}

\usepackage{setspace}

\usepackage{amsmath}

\usepackage{mathrsfs,amsfonts}

\usepackage{amsthm,amsxtra}

\usepackage[final]{showkeys}

\usepackage{stmaryrd}

\usepackage{algorithm}
\usepackage{algorithmicx}
\usepackage{algpseudocode}
\usepackage{mathtools}

\usepackage{yhmath}

\newif\ifPDF
\ifx\pdfoutput\undefined
\PDFfalse
\else
\ifnum\pdfoutput > 0
\PDFtrue
\else
\PDFfalse
\fi
\fi

\ifPDF
\usepackage{pdftricks}
\begin{psinputs}
	\usepackage{pstricks}
	\usepackage{pstcol}
	\usepackage{pst-plot}
	\usepackage{pst-tree}
	\usepackage{pst-eps}
	\usepackage{multido}
	\usepackage{pst-node}
	\usepackage{pst-eps}
\end{psinputs}
\else
\usepackage{pstricks}
\fi

\ifPDF
\usepackage[debug,pdftex,colorlinks=true, 
linkcolor=blue, bookmarksopen=false,
plainpages=false,pdfpagelabels]{hyperref}
\else
\usepackage[dvips]{hyperref}
\fi

\usepackage[openbib]{currvita}

\usepackage{fancyhdr}

\newtheorem{theorem}{Theorem}[section]
\newtheorem{lemma}[theorem]{Lemma}
\newtheorem{definition}[theorem]{Definition}

\newtheorem{remark}[theorem]{Remark} 
\newtheorem{corollary}[theorem]{Corollary}

\newcommand{\dint}{\displaystyle\int}
\newcommand{\eps}{\varepsilon}

\newcommand{\bbR}{\mathbb R} \newcommand{\bbS}{\mathbb S}

 \newcommand{\bn}{\mathbf n}

\newcommand{\cA}{\mathcal A} \newcommand{\cB}{\mathcal B}
\newcommand{\cC}{\mathcal C} \newcommand{\cD}{\mathcal D} 
\newcommand{\cE}{\mathcal E} 
 \newcommand{\cH}{\mathcal H}
\newcommand{\cI}{\mathcal I} 
\newcommand{\cK}{\mathcal K} 
\newcommand{\cM}{\mathcal M} 
\newcommand{\cO}{\mathcal O} \newcommand{\cP}{\mathcal P} 
 
\newcommand{\cS}{\mathcal S} \newcommand{\cT}{\mathcal T}
 
 \newcommand{\cX}{\mathcal X}

\newcommand{\aver}[1]{\langle {#1} \rangle}

\newenvironment{keywords}
{\noindent{\bf Key words.}\small}{\par\vspace{1ex}}

\makeatletter
\newcommand{\chapterauthor}[1]{%
	{\parindent0pt\vspace*{-25pt}%
		\linespread{1.1}\large\scshape#1%
		\par\nobreak\vspace*{35pt}}
	\@afterheading%
}
\makeatother

\title{Forward and inverse problems of a semilinear transport equation}
\author{Kui Ren\thanks{Department of Applied Physics and Applied Mathematics, Columbia University, New York, NY 10027;
kr2002@columbia.edu}\and Yimin Zhong\thanks{Department of Mathematics and Statistics, Auburn University, Auburn, AL 36830;
yimin.zhong@auburn.edu}}
\date{}

\begin{document}
\maketitle

\begin{abstract}
We study forward and inverse problems for a semilinear radiative transport model where the absorption coefficient depends on the angular average of the transport solution. Our first result is the well-posedness theory for the transport model with general boundary data, which significantly improves previous theories for small boundary data. For the inverse problem of reconstructing the nonlinear absorption coefficient from internal data, we develop stability results for the reconstructions and unify an $L^1$ stability theory for both the diffusion and transport regimes by introducing a weighted norm that penalizes the contribution from the boundary region. The problems studied here are motivated by applications such as photoacoustic imaging of multi-photon absorption of heterogeneous media.
\end{abstract}
\begin{keywords}
semilinear radiative transport, inverse problems, diffusion approximation, fixed-point theorem, spectral analysis, photoacoustic imaging, multi-photon absorption 
\end{keywords}

\section{Introduction}

  
The objective of this work is to study forward and inverse problems for a radiative transport model with a nonlinear absorption mechanism, for applications in imaging modalities such as quantitative photoacoustic tomography. 

To introduce the mathematical model, let $\Omega\subset \bbR^d$ ($d\ge 2$) be a convex domain with smooth boundary $\partial \Omega$, and denote $\bbS^{d-1}$ the unit sphere in $\bbR^d$, $X := \Omega\times \bbS^{d-1}$, and $\Gamma_{\pm} = \{(x, v)\in\partial \Omega\times \bbS^{d-1} \mid \pm n(x)\cdot v > 0\}$, where $n(x)$ is the unit outer normal vector at $x\in\partial\Omega$. We consider the following transport equation for the nonlinear absorption process~\cite{KrLiReScZh-SIAM23,ReZh-SIAM21, StZh-SIAM22}:
\begin{equation}\label{eq: mpa-rte}
\begin{aligned}
    v\cdot \nabla u (x, v) + ( \Sigma_a(|\aver{u}|) + \Sigma_s(x) ) u(x, v) &= \Sigma_s(x)\cK u(x, v) \quad &&\text{ in } X, \\
    u(x, v) &= f_{-}(x, v)\quad  &&\text{ on }\Gamma_{-},
\end{aligned}
\end{equation}
where $\Sigma_a: L^{\infty}_{+}(\Omega)\mapsto L^{\infty}_{+}(\Omega)$ is a nonlinear functional and $\Sigma_s(x)\in L^{\infty}_{+}(\Omega)$ denote the nonlinear absorption operator and the scattering coefficients, respectively. 
$\aver{u}$ is the averaged density of $u(x, v)$ over the angular variable $v\in\bbS^{d-1}$, 
\begin{equation*}
    \aver{u} := \int_{\bbS^{d-1}}  u(x, v) dv 
\end{equation*}
with $dv$ being the normalized surface measure on $\bbS^{d-1}$. The absolute value in $|\aver{u}|$ is to ensure well-posedness.  $f_{-}(x, v) \ge 0$ is the incoming illumination source function on $\Gamma_{-}$. The linear operator $\cK$ is defined by 
\begin{equation*}
    \cK u(x, v) = \int_{\bbS^{d-1}} p(x, v, v') u(x, v') dv',
\end{equation*}
where $p(x, v', v) \ge 0$ is the scattering phase function at point $x$ such that
\begin{equation*}
    \int_{\bbS^{d-1}} p(x, v, v') dv' = \int_{\bbS^{d-1}} p(x, v, v') dv = 1.
\end{equation*}

Radiative transport equations similar to~\eqref{eq: mpa-rte} are useful in modeling the propagation of particles in heterogeneous media where the absorption strength of the media depends on the local density of the particles. Such a situation appears in the propagation of near-infrared photons in biological tissues with multiphoton absorption~\cite{BaReZh-JBO18,Mahr-QE12,RuPe-AOP10,YiLiAl-AO99}.

The first goal of this work is to improve existing well-posedness theory for the radiative transport model~\eqref{eq: mpa-rte}. Existing theory in~\cite{ReZh-SIAM21,StZh-SIAM22} assumes that the boundary datum $f_-$ is sufficiently small. This assumption significantly limits the applicability of the theory, as nonlinear absorption mostly happens when the density of the particles is very high. We will remove this smallness assumption in our theory in Section~\ref{SEC:Forw}.

In applications such as quantitative photoacoustic imaging of two-photon absorption~\cite{DeStWe-Science90,MoViDi-APL08,XuWaXuZhYi-OC21,YaNaTa-OE11,YeKo-OE10}, one is interested in reconstructing the absorption property of~\eqref{eq: mpa-rte} from data related to its solutions. The second goal of this work is to analyze the problem of reconstructing a generic multi-photon absorption model where $\Sigma_a$ is a homogeneous polynomial of the local density $\aver{u}$ from the total absorbed energy data. We will develop some stability theories for the inverse problem in Section~\ref{SEC:IP Abso}, and then unify the stability theories of the diffusion regime and the transport regime in Section~\ref{Sec: trans}.

\section{The forward problem}
\label{SEC:Forw}
We start with the well-posedness of the transport equation~\eqref{eq: mpa-rte}. To distinguish from the existing results in~\cite{ReZh-SIAM21, StZh-SIAM22}, we emphasize that our well-posedness theory does not make further assumptions about the smallness of the illumination source $f_{-}$. This is a significant improvement over existing results. 

Let us denote $X = \Omega\times \bbS^{d-1}$. We introduce the space $\cH^p(X)$
\begin{equation*}
    \cH^p(X) = \{f(x, v)\mid f\in L^p(X),\text{and } v\cdot \nabla f\in L^p(X)\},
\end{equation*}
with $L^p(X)$ the standard $L^p$ space of functions on $X$. We define $L^p(\Gamma_{-})$ to be the trace space of $\cH^p$ on $\Gamma_{-}$ that is equipped with the norm $\|f\|_{L^p(\Gamma_{-})}$:
\begin{equation*}
   \|f\|_{L^p(\Gamma_{-})} := \left(\int_{\partial\Omega} \int_{v\cdot n(x) < 0} |n(x) \cdot v| |f(x, v)|^p dv ds \right)^{1/p}.
\end{equation*}

In the rest of the work, we assume the coefficients and the source function $f_{-}$ satisfy the following conditions: 
\begin{enumerate}
    \item [($\cA$).] (Boundedness) There exist constants $\underline{\Sigma}_a$, $\overline{\Sigma}_s$, $\underline{f}$, $\overline{f}$ and a non-decreasing function $g\in L^{\infty}(\bbR_{+})$ that 
    \begin{equation*}
    \begin{aligned}
&0 <  \underline{\Sigma}_a \le \Sigma_a(m) \le g(\|m\|_{\infty}),\quad &&\forall (x, m)\in \Omega\times L^{\infty}_{+}(\Omega),\\ & 0\le \Sigma_s(x) \le \overline{\Sigma}_s,\quad &&\forall x\in \Omega,
    \end{aligned}
    \end{equation*}
    and $0 < \underline{f} \le f_{-}(x, v) \le \overline{f}$ on $\Gamma_{-}$.
    \item [($\cB$).] (Continuity) $\Sigma_a(\cdot)$ is continuous under $L^2$ metric, that is, $\|m_1 - m_2\|_{L^2(\Omega)}\to 0$ implies $ \| \Sigma_a(m_1) - \Sigma_a(m_2) \|_{L^{2}(\Omega)} \to 0$. 
    \item [($\cC$).] (Positive-definiteness) The nonlinear absorption coefficient $\Sigma_a(\cdot)$ is continuously Fr\'echet differentiable with the Fr\'echet derivative, denoted by $\Sigma_a'$, satisfying
    \begin{equation*}
        \int_{\Omega} \Sigma'_a[f](x)\cdot f(x) dx \ge 0.
    \end{equation*}
    The equal sign holds if and only if $f\equiv 0$.
\end{enumerate}

Let us mention that, in the multi-photon absorption model we will study in Section~\ref{SEC:IP Abso}, the absorption coefficient $\Sigma_a(|\aver{u}|)$ is given in a polynomial form, such as in~\eqref{EQ:MPA Model} and~\eqref{EQ:MPA Model2}. It is straightforward to verify that such absorption models satisfy the conditions $(\cA)$-$(\cC)$.

\subsection{Existence}

The existence of a solution in $\cH^{\infty}(X)$ to~\eqref{eq: mpa-rte} can be proved through a similar idea from~\cite{ReZh-SIAM21}. The following simple estimate is based on linear transport theory~\cite{dautray2012mathematical_v6}.

\begin{lemma}\label{lem: bound}
    Suppose the condition $(\cA)$ is satisfied, and let $m\in L^{\infty}_{+}(\Omega)$ be such that $0\le m \le \|f_{-}\|_{L^{\infty}(\Gamma_{-})}$, then the solution $u\in \cH^{\infty}(X)$ to
    \begin{equation}\label{eq: m-rte}
\begin{aligned}
    v\cdot \nabla u (x, v) + ( \Sigma_a(m) + \Sigma_s(x) ) u(x, v) &= \Sigma_s(x)\cK u(x, v) \quad &&\text{ in } X, \\
    u(x, v) &= f_{-}(x, v)\quad  &&\text{ on }\Gamma_{-}.
\end{aligned}
\end{equation}
    satisfies
    \begin{equation*}
       0 < \mathsf{c} \le u \le \mathsf{C},
    \end{equation*}
    where the constants $ \mathsf{C} = \|f_{-}\|_{L^{\infty}(\Gamma_{-})}$ and  $\mathsf{c} = \underline{f} e^{-\operatorname{diam}(\Omega) g( \mathsf{C})}$.
\end{lemma}
\begin{proof}
    Suppose $u\in \cH^{\infty}(X)$ is the solution to~\eqref{eq: m-rte}, then $\mathsf{C} \ge u \ge 0$ from the classical linear transport theory~\cite{dautray2012mathematical_v6}. Let $\widetilde{u}$ be the solution to the following linear transport equation
    \begin{equation*}
    \begin{aligned}
            v\cdot \nabla \widetilde{u}(x, v) + \Sigma_a( m) \widetilde{u} &= 0\quad &&\text{ in }X, \\
            \widetilde{u}(x, v) &= f_{-}(x, v) \quad &&\text{ on }\Gamma_{-}\,.
    \end{aligned}
    \end{equation*}
    Then, from the analytic expression of $\widetilde{u}$, we can verify that $\widetilde{u} \ge \underline{f} e^{-\operatorname{diam}(\Omega) g(\|m\|_{\infty})} \ge \mathsf{c}$. We also verify that $\phi = u - \widetilde{u}$ satisfies
    \begin{equation*}
    \begin{aligned}
        v\cdot \nabla \phi(x, v) + (\Sigma_a( m ) + \Sigma_s(x))\phi(x, v) &= \Sigma_s(x) \cK u \quad &&\text{ in }X, \\
        \phi(x, v) &= 0 \quad &&\text{ on }\Gamma_{-}.
    \end{aligned}
    \end{equation*}
    Thus $\phi \ge 0$, which implies $u \ge \widetilde{u} \ge \mathsf{c}$.
\end{proof}

\begin{theorem}
    Suppose the conditions $(\cA)$ and $(\cB)$ are satisfied, then there exists a solution $u\in \cH^{\infty}$ to~\eqref{eq: mpa-rte}.
\end{theorem}
\begin{proof}
We define the mapping $\cS: L^{\infty}_{+}(\Omega)\mapsto L^{\infty}_{+}(\Omega)$ through the relation 
\begin{equation}\label{EQ:S}
\cS(m) = \aver{u},
\end{equation}
where $u \in \cH^{\infty}(X)$ denotes the unique solution to
\begin{equation}\label{eq: mapping}
\begin{aligned}
    v\cdot \nabla u(x, v) + (\Sigma_a( m) + \Sigma_s(x)) u &= \Sigma_s(x)\cK u(x, v)  \quad &&\text{ in } X,\\
    u(x, v) &= f_{-}(x, v) \quad &&\text{ on }\Gamma_{-}. 
\end{aligned}
\end{equation}
When $m\in L^{\infty}_{+}(\Omega)$, from the condition $(\cA)$ we have that $\Sigma_a( m) \ge \underline{\Sigma}_a > 0$. Therefore, $0 \le u \le \|f_{-}\|_{\infty}$. Let $\cM$ be the set of bounded functions with $L^2(\Omega)$ topology:
\begin{equation*}
    \cM = \{m \in L_{+}^{\infty}(\Omega) \mid 0 \le m \le \|f_{-}\|_{\infty} \}
\end{equation*}
Therefore, the mapping $\cS(\cM)\subset \cM$.  It is straightforward to see that $\cM$ is convex and closed. We then verify that the mapping $\cS$ is continuous on $\cM$. Let $m_1, m_2\in \cM$ and denote by $u_i$ the solution to~\eqref{eq: mapping} with the absorption coefficient $\Sigma_a( m_i)$, $i=1,2$. Then, $\phi = u_1 - u_2$ solves  
\begin{equation*}
\begin{aligned}
    v\cdot \nabla \phi + (\Sigma_a( m_1) + \Sigma_s ) \phi &= \Sigma_s \cK \phi(x, v) - (\Sigma_a( m_1) - \Sigma_a( m_2)) u_2 \quad &&\text{ in }X, \\
    \phi(x, v) &= 0 \quad&&\text{ on }\Gamma_{-}.
\end{aligned}
\end{equation*}
Multiply the above equation by $\phi$ and integrate over $X$, we find that 
\begin{equation*}
\begin{aligned}
 &\int_{X} \Sigma_a( m_1) |\phi|^2 dx dv \\ & \le    \int_{\Gamma_{+}}\frac{1}{2} v\cdot n |\phi|^2 dv dS + \int_{X} \Sigma_a( m_1) |\phi|^2 dx dv + \int_{X} \Sigma_s(x) (|\phi|^2 - \phi \cK \phi) dx dv \\&= -\int_X  (\Sigma_a( m_1) - \Sigma_a( m_2)) u_2 \phi dx dv \\
    &\le \|\Sigma_a( m_1) - \Sigma_a( m_2) \|_{L^2(X)} \|u_2\|_{L^{\infty}(X)} \|\phi\|_{L^2(X)}.
\end{aligned}
\end{equation*}
Using condition $(\cB)$, we conclude that if $\|m_1 - m_2\|_{L^2(\Omega)}\to 0$, then
\begin{equation*}
    \|\phi\|_{L^2(X)} \le \frac{1}{\underline{\Sigma}_a}\|  (\Sigma_a( m_1) - \Sigma_a( m_2))  \|_{L^2(X)} \|u_2\|_{L^{\infty}(X)} \to 0.
\end{equation*}
Therefore, the continuity of $\cS$ is obtained by applying the Cauchy-Schwarz inequality, 
\begin{equation*}
    \|\cS(m_1) - \cS(m_2)\|_{L^2(\Omega)} = \|\aver{u_1} - \aver{u_2}\|_{L^2(\Omega)} \le \|\phi\|_{L^2(X)}. 
\end{equation*}
Finally, the Averaging Lemma~\cite{golse1988regularity} shows that $\cS(m)\in W^{1/2, 2}(\Omega)$ for any $m\in \cM$, which compactly embeds into $L^2(\Omega)$. Therefore, the Schauder fixed-point theorem implies that there exists a fixed point $m^{\ast}\in \cM$ that $\cS(m^{\ast}) = m^{\ast}$. The solution $u^{\ast}\in \cH^{\infty}(X)$ then exists by solving~\eqref{eq: mapping} with $m$ replaced with $m^{\ast}$.  
\end{proof}

\subsection{Uniqueness}

The earlier uniqueness results in~\cite{ReZh-SIAM21, KrLiReScZh-SIAM23, StZh-SIAM22} rely on the assumption that the source function $f_{-}$ must be small in a certain way. We will lift this assumption. The key ingredient is the uniqueness theorem~\cite{kellogg1976uniqueness} for Schauder's fixed-point theory, which was also used in~\cite{ReZh-SIAM21}.
\begin{theorem}[Kellogg]\label{thm: kellogg}
Let $\cM$ be a bounded convex open subset of a real Banach space $\cX$, and $F: \cM \mapsto \cM$ be a compact continuous map which is continuously Fr\'echet
differentiable on $\cM$. If (i) for each $m\in\cM$, $1$ is not an eigenvalue of $F'(m)$, and (ii) for
each $m\in\partial\cM$, $m \neq  F (m)$, then $F$ has a unique fixed point in $\cM$.
\end{theorem}

This theorem was also contained in the monograph of M. S. Berger~\cite{berger1977nonlinearity}. The uniqueness theorem was then further extended by Smith and Stuart to more general cases in~\cite{smith1980uniqueness}. We will need the following inequality to proceed. 

\begin{lemma}\label{lem: am-gm}
Suppose $\underline{u} := \inf_{v\in\bbS^{d-1}} u(v) > 0$ and define an integral operator $\cK: L^2(\bbS^{d-1})\mapsto L^2(\bbS^{d-1})$ as
\begin{equation*}
    \cK f(v) := \int_{\bbS^{d-1}} p(v, v') f(v') dv', 
\end{equation*}
where $p(v,v') \ge 0$ is a scattering phase function. Then for any $\phi\in L^2(\bbS^{d-1})$,
    \begin{equation*}
 \int_{\bbS^{d-1}}\frac{|\cK \phi|^2}{\cK u} dv\le  \int_{\bbS^{d-1}}\frac{\phi^2}{u} dv.
\end{equation*}
\end{lemma}
\begin{proof}
Using the Cauchy-Schwarz inequality, we have
\begin{equation*}
   \left( \int_{\bbS^{d-1}} p(v,v')\frac{\phi^2(v)}{u(v)} dv \right) \left(\int_{\bbS^{d-1}} p(v, v') u(v) dv \right) \ge \left(\int_{\bbS^{d-1}} \phi(v) p(v,v') dv \right)^2 .
   \end{equation*}
Therefore,
\begin{equation*}
\int_{\bbS^{d-1}} p(v, v') \frac{\phi^2(v)}{u(v)} dv \ge \frac{|\cK \phi(v')|^2}{\cK u(v')}.
\end{equation*}
We obtain our result by taking the integral of $v'$ on both sides.
\end{proof}

\begin{theorem}
   If conditions $(\cA)$, $(\cB)$ and $(\cC)$ are satisfied, the $\cH^\infty$ solution to~\eqref{eq: mpa-rte} is unique.
\end{theorem}
\begin{proof}
   We verify the condition (ii) in Theorem~\ref{thm: kellogg}. According to condition $(\cA)$, the absorption coefficient $\Sigma_a( m)$ is strictly positive for any $m\in \cM$, therefore $\aver{u} < \|f_{-}\|_{\infty}$ on $\Omega$. This means that $\cS$, defined in~\eqref{EQ:S}, does not have any fixed-point on $\partial\cM$. 
   
   The remaining task is to check condition (i). Let us assume that there exists $m'\in\cM$ such that $1$ is an eigenvalue of the Fr\'echet derivative of $\cS(m)$ at $m = m'$.  
   Let $u$ and $\Tilde{u}$ be the solutions to~\eqref{eq: m-rte} at $m = m'$ and $m = m' + \delta m \in \cM$. Then after linearization,  $\delta u = \tilde{u} - u$ satisfies 
   \begin{equation*}
       v\cdot \nabla \delta u + \Sigma_a(m) \delta u + \Sigma_s(\cI - \cK) \delta u = -\Sigma_a'(m) [\delta m] u \quad \text{ in } X,
   \end{equation*}
   where $\Sigma_a'(m)$ is the Frech\'et derivative of $\Sigma_a$. 
   Since $1$ is an eigenvalue of the linear mapping $\delta m\to \aver{\delta u}$ at $m = m'$, there is a nonzero perturbation $\delta m$ such that $\aver{\delta u} = \delta m$. 
   Then the following transport equation permits a non-trivial solution $\phi$:
\begin{equation}\label{eq: t0}
\begin{aligned}
    v\cdot \nabla \phi + \Sigma_a( m)\phi + \Sigma_s (\cI - \cK) \phi &= - \Sigma'_a(m)\left[\aver{\phi}\right] u, &&\quad \text{ in }X, \\
    \phi &= 0, &&\quad\text{ on }\Gamma_{-},
\end{aligned}
\end{equation}
where $u$ is the associated solution to~\eqref{eq: mpa-rte} for $m\in\cM$. Lemma~\ref{lem: bound} then implies that $u$ is bounded from below. The following two identities are easy to verify:
\begin{equation}\label{eq: t1}
    \int_X ( v\cdot \nabla \phi ) \frac{\phi}{u} d x dv = \int_X v\cdot \nabla \frac{|\phi|^2}{2u} dx dv - \int_X ( v\cdot \nabla \frac{1}{u} ) \frac{|\phi|^2}{2} dx dv  
\end{equation}
and
\begin{equation}\label{eq: t2}
    v\cdot \nabla \frac{1}{u} = \frac{\Sigma_a( m)}{u} + \frac{\Sigma_s(\cI - \cK)u}{u^2}.
\end{equation}
Multiply the equation~\eqref{eq: t0} by $\dfrac{\phi}{u}$ and integrate on $X$, then use~\eqref{eq: t1}, we obtain
\begin{equation*}
\begin{aligned}
    &\int_X v\cdot \nabla \frac{|\phi|^2}{2u} dx dv - \int_X ( v\cdot \nabla \frac{1}{u} ) \frac{|\phi|^2}{2} dx dv  + \int_X \Sigma_a(m)\frac{\phi^2}{u} dx dv  + \int_X (\Sigma_s(\cI - \cK)\phi ) \frac{\phi}{u} dx dv \\
    &= -\int_{\Omega} \Sigma_a'(m)[\aver{\phi}]\aver{\phi} dx \le 0.
\end{aligned}
\end{equation*}
The last inequality is from condition $(\cC)$.
The first term on the left-hand side is non-negative by the divergence theorem. Thus, 
\begin{equation*}
     - \int_X ( v\cdot \nabla \frac{1}{u} ) \frac{|\phi|^2}{2} dx dv  + \int_X \Sigma_a(m)\frac{\phi^2}{u} dx dv  + \int_X \Sigma_s(\cI - \cK)\phi \frac{\phi}{u} dx dv \le 0.
\end{equation*}
Using the identity~\eqref{eq: t2}, we obtain
\begin{equation*}
       - \int_X \left( \frac{\Sigma_a( m)}{u} + \frac{\Sigma_s(\cI - \cK)u}{u^2} \right) \frac{|\phi|^2}{2} dx dv  + \int_X \Sigma_a( m)\frac{\phi^2}{u} dx dv  + \int_X \Sigma_s(\cI - \cK)\phi \frac{\phi}{u} dx dv \le 0.
\end{equation*}
A slight rearrangement gives 
\begin{equation*}
    \int_X ( \Sigma_a( m)  + \Sigma_s ) \frac{\phi^2}{2u}dx dv + \int_X \Sigma_s\left( \frac{\cK u}{u^2} \frac{\phi^2}{2} - \phi \frac{\cK \phi}{u} \right)dx dv \le 0.
\end{equation*}
Using AM-GM inequality, we find that 
\begin{equation*}
   \int_X \Sigma_s\left( \frac{\cK u}{u^2} \frac{\phi^2}{2} - \phi \frac{\cK \phi}{u} \right)dx dv \ge -\frac{1}{2} \int_X \Sigma_s \left|\frac{\cK \phi}{\sqrt{\cK {u}}}\right|^2 dx dv.
\end{equation*}
Then we can derive the following inequality,
\begin{equation*}
    \int_X \Sigma_a( m) \frac{\phi^2}{2u}dx dv + \frac{1}{2} \int_D \Sigma_s \int_{\bbS^{d-1}}\left(  \frac{\phi^2}{u} - \left|\frac{\cK \phi}{\sqrt{\cK {u}}}\right|^2 \right) dv dx \le 0.
\end{equation*}
The second part on the left-hand side is non-negative using Lemma~\ref{lem: am-gm}. Therefore,
\begin{equation*}
        \int_X \Sigma_a( m)  \frac{\phi^2}{2u}dx dv \le 0,
\end{equation*}
which leads to $\phi=0$, a contradiction. Hence, the fixed point is unique.
\end{proof}

\section{Reconstruction of absorption}
\label{SEC:IP Abso}

We now study an inverse coefficient problem to the radiative transport model~\eqref{eq: mpa-rte}. We aim at reconstructing the absorption coefficient $\Sigma_a$ from internal data. In particular, we are interested in the application of such inverse problems in quantitative photoacoustic imaging of multi-photon absorption. 

The multi-photon absorption model we consider takes the form~\cite{KrLiReScZh-SIAM23, ReZh-SIAM21}
\begin{equation}\label{EQ:MPA Model}
    \Sigma_a(|\aver{u}|) := \sum_{k=0}^K \sigma_{a, k}(x) |\aver{u}|^k,
\end{equation}
where each $\sigma_{a, k}(x)\ge 0$ ($0\le k\le K$) is uniformly bounded and is proportional to the probability that a molecule at point $x$ gets excited by absorbing $(k+1)$ photons simultaneously. The following generalized version of this absorption coefficient can also be considered
\begin{equation}\label{EQ:MPA Model2}
    \Sigma_a( |\aver{u}|) = \sum_{k=0}^{K} \sigma_{a, k}(x)( \cT_k |\aver{u}|)^k
\end{equation}
where $\cT_k: L_{+}^{\infty}(\Omega)\mapsto L_{+}^{\infty}(\Omega)$ is a positive integral operator defined by a positive-definite kernel $T_k$: 
\begin{equation*}
    \cT_k g(x) = \int_{\Omega} T_k(x, y) g(y) dy.
\end{equation*}

The following corollary ensures the uniqueness of the solution to the forward problem with such absorption coefficients. It can be proven by verifying that the conditions $(\cA)$, $(\cB)$, and $(\cC)$ are satisfied by absorption coefficients~\eqref{EQ:MPA Model} and~\eqref{EQ:MPA Model2} under the conditions in the corollary.
\begin{corollary}
    Assume that $\sigma_{a, 0}(x) > 0$ and $\sigma_{a, k}(x) \ge 0$ ($1\le k\le K$). Then the radiative transport equation~\eqref{eq: mpa-rte} with the absorption coefficient~\eqref{EQ:MPA Model} or~\eqref{EQ:MPA Model2} admits a unique solution in $\cH^{\infty}(X)$.
\end{corollary}
 

In our inverse problems of reconstructing $\Sigma_a$, we assume that we have access to a finite number of data of the following form
\begin{equation}\label{EQ:IP Data}
H = \Sigma_a(\aver{u})\aver{u}
\end{equation}
where $u$ is the solution to the radiative transport model with coefficient $\Sigma_a$. Physically, $H$ is the total absorbed energy inside the domain of photon propagation~\cite{KrLiReScZh-SIAM23}.

The reconstruction of the absorption coefficient of radiative transport models from the interior measurement data of the form $H$ has been studied from different perspectives; see, for instance,~\cite{BaJoJu-IP10,ReZh-SIAM21, zhao2021quantitative, mamonov2012quantitative,ding2015one, ren2013hybrid,LaReZh-SIAM22} and references therein for some samples of recent progress in the field. Related inverse problems based on other types of data for similar transport models have also been extensively studied; see~\cite{Bal-IP09,ChSt-IP96,ChGaLiWa-arXiv25,HeKlLiTa-SIAM24,GaLiNa-SIAM22,ChSc-SIAM17,LaLiUh-SIAM19, KlPa-JMAA08, KlLiYa-IP23,LiSu-IP20,MaYa-IP14,LaYa-SIAM24,LiOu-PAMS23,LaZh-arXiv24,StUh-MAA03, SuLi-SIAM22,Wang-AIHP99} and references therein.

\subsection{Diffusion regime}
\label{sec: l1}

The inverse problem simplifies in the diffusive regime. Following a formal derivation in~\cite{KrLiReScZh-SIAM23}, the limiting nonlinear diffusion model becomes
\begin{equation}\label{eq: diff}
\begin{aligned}
    -\nabla \cdot \left(D \nabla U\right) + \Sigma_a( U) U &= 0,\quad &&\text{ in }\Omega\\
    U &= f(x) \quad &&\text{ on }\partial\Omega, 
\end{aligned}
\end{equation}
where $D$ and $\Sigma_a$ are the diffusion and absorption coefficients. The internal datum $H$ now becomes 
\begin{equation*}
    H = \Sigma_a(U) U\,.
\end{equation*}
This simplification of the forward model and the internal data allows us to derive, in a straightforward manner, a stability estimate in $L^{\infty}$ norm (assuming the unique solution $U\in L^{\infty}(\Omega)$); see~\cite{ReZh-SIAM18}.  Let us point out that besides the $L^{\infty}$ stability estimate, the other $L^p$ ($1\le p< \infty$) stability estimates also hold for the nonlinear diffusion model~\eqref{eq: diff}, but are often overlooked.

\begin{theorem}\label{thm: diff-pat}
    Let $U$ and $\widetilde U$ be respectively the unique $L_+^\infty(\Omega)\cap W^{1,\infty}(\Omega)$ solutions to the diffusion model~\eqref{eq: diff} with absorption coefficients $\Sigma_a$ and $\widetilde \Sigma_a$ that satisfy the condition $(\cA)$. Assume further that $f$ and $\Omega$ are such that $U, \widetilde U\ge c>0$ for some $c$. Then, $\forall p \in [1, \infty)$, there exists a constant $C$  that
    \begin{equation*}
        \|\Sigma_a(U) - \widetilde{\Sigma}_a(\widetilde{U})\|_{L^p(\Omega)} \le C \|H - \widetilde{H}\|_{L^p(\Omega)}\,,
    \end{equation*}
    where $H$ and $\widetilde H$ are data corresponding to $\Sigma_a$ and $\widetilde \Sigma_a$ respectively. The constant $C$ depends on $\|f\|_{\infty}$.
\end{theorem}
\begin{proof}
    We refer to~\cite{KrLiReScZh-SIAM23,ReZh-SIAM18} for more details on the study of the diffusion model~\eqref{eq: diff} and the conditions on $f$ under which $U, \widetilde U\ge c>0$ for some $c$.
    
    To prove the statement, let $\delta U = U - \widetilde{U}$. It is clear then $\delta U$ solves
    \begin{equation*}
 -\nabla \cdot \left(D \nabla \delta U\right) = - (H - \widetilde{H}).
    \end{equation*}
    Multiplying $|\delta U|^{p-1} \text{sgn}(\delta U)$ to both sides and integrating over $\Omega$ then gives us (using $p' = \frac{p}{p-1}$),
    \begin{equation*}
    \begin{aligned}
       \frac{1}{p} \int_{\Omega} D |\nabla \delta U|^p dx &\le \int_{\Omega}  | H - \widetilde{H}| \cdot  |\delta U|^{p-1} dx \le \|H - \widetilde{H} \|_{L^p(\Omega)} \||\delta U|^{p-1}\|_{L^{p'}(\Omega)} \\
        &\le \|H - \widetilde{H}  \|_{L^p(\Omega)} \|\delta U\|_{L^p(\Omega)}^{p-1}.
    \end{aligned}
    \end{equation*}
    By the classical Poincar\'e inequality, there exists a constant $C'$ that $$\|\delta U\|_{L^p(\Omega)} \le C_{p,\Omega} \|\nabla \delta U\|_{L^p(\Omega)} \le C' \| H - \widetilde{H} \|_{L^p(\Omega)},$$ where $C_{p,\Omega}$ is the Poincar\'e constant. Next, we observe that
    \begin{equation*}
       \Sigma_a(U) - \widetilde{\Sigma}_a(\widetilde{U})  = \frac{H - \widetilde{H} -\delta U \widetilde{\Sigma}_a(\widetilde{U})}{U}\,.
    \end{equation*}
    Using this, and the fact that $U\ge c>0$, we obtain that
    \begin{equation*}
          \| \Sigma_a(U) - \widetilde{\Sigma}_a(\widetilde{U}) \|_{L^p(\Omega)} \le C \| H - \widetilde{H}\|_{L^p(\Omega)}
    \end{equation*}
    for a certain constant $C$. The proof is complete.
\end{proof}

\subsection{Transport regime}

Because of the well-known connection between diffusion and transport models, it seems natural to ask for a similar stability estimate for the transport model. However, due to the difficulty caused by the anisotropy in the solution to the transport model, the existing efforts in finding such stability estimates require either restrictions in data~\cite{zhao2021quantitative,ReZh-SIAM21} or arbitrarily many measurements through a linearization technique~\cite{LaReZh-SIAM22}. In this section, we will show the $L^1$ stability estimate for the transport model without making further assumptions other than condition $(\cA)$. 

\begin{theorem}\label{thm: pat}
Let $u$ and $\widetilde{u}$ be the corresponding solutions to~\eqref{eq: m-rte} with absorption coefficients $\Sigma_a$ and $\widetilde{\Sigma}_a$, respectively. The associated data are denoted as $H$ and $\widetilde{H}$, respectively.  Assume further that $\Sigma_a$ and $\Sigma_a$ satisfy the condition $(\cA)$. Then $H = \widetilde{H}$ implies $\Sigma_a( \aver{u}) = \widetilde{\Sigma}_a( \aver{\widetilde{u}})$. Moreover, there exists a constant $C > 0$ such that 
\begin{equation*}
  \left\| \frac{ \Sigma_a( \aver{u}) - \Sigma_a( \aver{\widetilde{u}}) } {\Sigma_a(\aver{u})}\right\|_{L^1(\Omega)} \le C \left \|\frac{H - \widetilde{H}}{\Sigma_a(\aver{u})}\right \|_{L^1(\Omega)}.
\end{equation*}
The constant $C$ depends on $\|f_{-}\|_{\infty}$.
\end{theorem}

\begin{proof}
Denote $\delta \Sigma_a = \Sigma_a( \aver{u}) - \widetilde{\Sigma}_a(\aver{\widetilde{u}})$, and $\delta u = u - \widetilde{u}$. Let $\Sigma_t(x) := \Sigma_a( \aver{u}) + \Sigma_s(x)$. Then the perturbation $\delta u$ satisfies the transport equation
\begin{equation}\label{eq: delta u}
  \begin{aligned}
         v\cdot \nabla \delta u + \Sigma_t(x) \delta u - \Sigma_s(x)\cK \delta u &= -\delta\Sigma_a \widetilde{u}\quad &&\text{ in }X, \\
         \delta u &= 0\quad &&\text{ on }\Gamma_{-}.
  \end{aligned}
\end{equation}
We verify also that $H - \widetilde{H} = \Sigma_a( \aver{u}) \aver{\delta u} + \delta\Sigma_a \aver{\widetilde{u}}$.
Therefore, we have
    \begin{equation*}
        |\aver{\delta u}| = \left|\frac{\delta\Sigma_a}{\Sigma_a( \aver{u})} \aver{\widetilde{u}}- \frac{H - \widetilde{H}}{\Sigma_a( \aver{u})}\right| \ge \left|\frac{\delta\Sigma_a}{\Sigma_a( \aver{u})} \aver{\widetilde{u}}\right| - \left|\frac{H - \widetilde{H}}{\Sigma_a( \aver{u})}\right|.
    \end{equation*}
Let us denote by $\tau_{\mp}(x, v)$ the distance from $x$ to boundary traveling in $\mp v$ direction, that is, 
\begin{equation*}
    \tau_{\mp}(x, v) = \sup\left\{s\in \bbR \mid x \mp sv \in \Omega\right\},
\end{equation*}
and denote $E(x, x-sv)$ by
\begin{equation*}
    E(x, x-sv) := \exp\left(-\int_0^s \Sigma_t(x-tv)  dt \right).
\end{equation*}
We further introduce the constant $\mu \in (0, 1)$ as 
$$\mu := \sup_{(y, v)\in X} (1 - E(x, x- sv)) \le 1 - e^{-\operatorname{diam}(\Omega)(g(\|f_{-}\|_{\infty}) + \overline{\Sigma}_s)}. $$
Then, we have
\begin{equation}\label{eq: less}
     \Sigma_t\aver{|\delta u|} \ge \mu {\Sigma_s}\aver{|\delta u|} + \big( {\Sigma_a + (1 - \mu)\Sigma_s}\big) \left( \left|\frac{\delta\Sigma_a}{\Sigma_a( \aver{u})} \aver{\widetilde{u}}\right| - \left|\frac{H - \widetilde{H}}{\Sigma_a( \aver{u})}\right| \right).
\end{equation}
We now solve the equation~\eqref{eq: delta u} along the characteristics and integrate over $\bbS^{d-1}$ to get
\begin{equation*}
\begin{aligned}
  \aver{|\delta u|}(x) &= \int_{\bbS^{d-1}}|\delta u(x, v)|dv \\ &=  \int_{\bbS^{d-1}} \left| \int_0^{\tau_{-}(x, v)} E(x, x-sv)   \left(- \delta\Sigma_a(x-sv) \widetilde{u}(x-sv, v) + \Sigma_s \cK \delta u (x-sv, v)\right) ds \right| dv\\
  &\le\int_{\bbS^{d-1}} \int_0^{\tau_{-}(x, v)} E(x, x-sv)   |-\delta\Sigma_a(x-sv) \widetilde{u}(x-sv, v)+ \Sigma_s \cK \delta u (x-sv, v) | ds dv.
\end{aligned}
\end{equation*}
We multiply the above inequality by $\Sigma_t$ and integrate over $\Omega$. Then, under the change of variable $y = x-sv$, the right-hand side becomes
\begin{equation}\label{eq: key}
\begin{aligned}
&\int_{\bbR^d}\int_{\bbS^{d-1}} \int_0^{\tau_{-}(x, v)} E(x, x-sv) \Sigma_t(x)  |-\delta\Sigma_a(x-sv)\widetilde{u}(x-sv, v) + \Sigma_s \cK \delta u  (x-sv, v)|ds dv dx \\
&=\int_{\bbR^d}\int_{\bbS^{d-1}} \int_0^{\tau_{+}(y, v)} E(y+ sv, y)  \Sigma_t(y+sv) |-\delta\Sigma_a(y) \widetilde{u}(y, v) + \Sigma_s \cK \delta u (y, v)|ds dv dy \\
&= \int_{\Omega}\int_{\bbS^{d-1}} \left(1 - E(y + \tau_{+}(y, v), y)\right) {|-\delta\Sigma_a(y) \widetilde{u}(y, v) + \Sigma_s \cK \delta u (y, v)|} dv dy\\
&\le\mu \int_{\Omega}\int_{\bbS^{d-1}} {|-\delta\Sigma_a(y) \widetilde{u}(y, v) + \Sigma_s \cK \delta u (y, v)|} dv dy \\
&\le \mu \int_{\Omega}\int_{\bbS^{d-1}} {|-\delta\Sigma_a(y) \widetilde{u}(y, v)|} dv dy + \mu \int_{\Omega}\int_{\bbS^{d-1}} {|\Sigma_s \cK \delta u (y, v)|} dv dy \\
&\le \mu \int_{\Omega} \left|{\delta\Sigma_a(y) \aver{\widetilde{u}}(y)}\right| dy + \mu \int_{\Omega} \left|{\Sigma_s(y)}\aver{|\delta u|}(y)\right| dy.
\end{aligned}
\end{equation}
Combining this with~\eqref{eq: less}, we get 
\begin{equation*}
\begin{aligned}
   \mu \int_{\Omega} \left|{\delta\Sigma_a(y) \aver{\widetilde{u}}(y)}\right| dy  &\ge \int_{\Omega} ({\Sigma_a + (1-\mu)\Sigma_s})\left( \left|\frac{\delta\Sigma_a}{\Sigma_a( \aver{u})} \aver{\widetilde{u}}\right| - \left|\frac{H - \widetilde{H}}{\Sigma_a( \aver{u})}\right| \right) dx.
\end{aligned}
\end{equation*}
Hence, we obtain the $L^1$ stability estimate
\begin{equation*}
\int_{\Omega} \left({(1-\mu)^{-1}\Sigma_a + \Sigma_s}\right) \left|\frac{H - \widetilde{H}}{\Sigma_a}\right| dy \ge  \int_{\Omega} \left|\Sigma_t(y) \frac{\delta\Sigma_a(y) \aver{\widetilde{u}}(y)}{\Sigma_a(y)}\right| dy.
\end{equation*}
The rest follows from the assumption that $\Sigma_a$ and $\widetilde \Sigma_a$ satisfy condition ($\cA$).
\end{proof}

The above result leads to the following corollary when $\Sigma_a$ takes the multi-photon absorption form~\eqref{EQ:MPA Model}:
\begin{equation}\nonumber
    \Sigma_a(|\aver{u}|) := \sum_{k=0}^K \sigma_{a, k}(x) |\aver{u}|^k,
\end{equation}
\begin{corollary}\label{COR: unique}
Let $\Sigma_a$ and $\widetilde \Sigma_a$ be absorption coefficients of the polynomial form.
Let $\{H_i\}_{i=0}^{K}$ and $\{\widetilde H_i\}_{i=0}^{K}$ be data associated with coefficients $\Sigma_a$ and $\widetilde \Sigma_a$, respectively, generated with illumination source $\{f_i\}_{i=0}^{K}$ satisfying the condition
\[
0 < f_{0} < f_{1} < \cdots < f_{K}\,, \quad \forall (x, v)\in\Gamma_{-}
\]
such that $$\sum_{k=1}^K  k \|f_K\|^{k}_{L^{\infty}(\Gamma_{-})} \sigma_{a, k}(x) \le \Sigma_s(x) \inf_{v, v'\in\bbS^{d-1}} p(x, v, v'),$$ where $p(x, v,v')$ is the scattering kernel of $\cK$.
Then 
\[
\{H_i\}_{i=0}^{K}=\{\widetilde H_i\}_{i=0}^{K} \implies
    \{\sigma_{a, k}\}_{k=0}^K =\{\widetilde \sigma_{a, k}\}_{k=0}^K\,.
\]
\end{corollary}
\begin{proof}
Let $u_i$ be the unique solution from source $f_{i}$, and define $\phi_j = u_{j+1} - u_{j}$ ($0\le j\le K-1$). Then we verify that $\phi_i$ satisfies the transport equation
\begin{equation*}
\begin{aligned}
    v\cdot \nabla \phi_j + \left(\Sigma_a(\aver{u_{j+1}}) + \Sigma_s \right)\phi_j &= \Sigma_s \cK \phi_j -  u_j \frac{\Sigma_a(\aver{u_{j+1}}) - \Sigma_a(\aver{u_{j}}) }{\aver{u_{j+1}} - \aver{u_{j}}} \aver{\phi_j}, \quad &&\text{ in } X, \\
    \phi_j &= f_{j+1} - f_{j} > 0, \quad &&\text{ on }\Gamma_{-}.
\end{aligned}
\end{equation*}
The second term on the right-hand side can be viewed as another scattering operator $\cP$ that
\begin{equation*}
    \cP \phi(x, v) = \int_{\bbS^{d-1}} \frac{\Sigma_a(\aver{u_{j+1}}) - \Sigma_a(\aver{u_{j}}) }{\aver{u_{j+1}} - \aver{u_{j}}} u_j(x, v) \phi(x, v') dv', 
\end{equation*}
Therefore, when $\Sigma_s(x)  p(x, v, v') - \frac{\Sigma_a(\aver{u_{j+1}}) - \Sigma_a(\aver{u_{j}}) }{\aver{u_{j+1}} - \aver{u_{j}}} u_j(x, v)  \ge 0$, we can write the whole right-hand side as a single scattering operator
\begin{equation*}
\cK' \phi(x, v) = \int_{\bbS^{d-1}}   \left( \Sigma_s(x)  p(x, v, v') - \frac{\Sigma_a(\aver{u_{j+1}}) - \Sigma_a(\aver{u_{j}}) }{\aver{u_{j+1}} - \aver{u_{j}}} u_j(x, v) \right)\phi(x, v') dv',
\end{equation*}
where the kernel is nonnegative from the assumption.
This leads to $\phi_j > 0$ ($0\le j\le K-1$) from the standard transport theory. Therefore, we have that
\begin{equation}\label{EQ:Monotone u}
   \aver{ u_{0} }  < \aver{ u_1 }  < \cdots < \aver{ u_{K} }.
\end{equation}
Meanwhile, by Theorem~\ref{thm: pat}, we have that $H_i=\widetilde H_i$ implies $\Sigma_a=\widetilde \Sigma_a$, and therefore $u_i=\widetilde u_i$ ($0\le i\le K$).

We therefore have the following system of equations for $\{\sigma_{a,k}-\widetilde \sigma_{a,k}\}_{k=0}^K$:
\[
    (\sigma_{a,0}-\widetilde \sigma_{a,0}) + (\sigma_{a,1}-\widetilde\sigma_{a,1})\aver{u_i}+\cdots+(\sigma_{a,K}-\widetilde \sigma_{a,K})\aver{u_i}^K=0,\quad 0\le i\le K
\]
By the monotonicity of $\aver{u_i}$ in~\eqref{EQ:Monotone u}, the system admits only zero solutions $\sigma_{a,k}-\widetilde \sigma_{a,k}=0$ ($0\le k\le K$). The conclusion then follows.
\end{proof}
\begin{remark}
    For the diffusion model~\eqref{eq: diff}, if the nonlinear absorption $\Sigma_a(U)$ follows the polynomial model~\eqref{EQ:MPA Model}, then the result in Corollary~\ref{COR: unique} also holds. This follows from the comparison principle~\cite{gilbarg_elliptic_2000}. To be precise, if $U_j$ and $U_{j+1}$ are two solutions with boundary source functions $0 < f_j < f_{j+1}$, respectively. Then $\phi = U_{j+1} - U_j$ satisfies the following elliptic equation 
    \begin{equation*}
        -\nabla (D(x) \nabla \phi) + \left( \Sigma_{a}(U_{j+1}) + C(x) U_j \right) \phi = 0,
    \end{equation*}
    where $C(x) = \frac{\Sigma_a(U_{j+1}) - \Sigma_a(U_j)}{U_{j+1} - U_j} \ge 0$, and $\phi > 0$ on the boundary. Therefore, $\phi = U_{j+1} - U_j > 0$ on the whole domain. However, the comparison (and maximum) principle does not hold for~\eqref{eq: m-rte} in the transport regime for large source functions. 
\end{remark}

We point out that the monotonicity is not necessary for the uniqueness result in Corollary~\ref{COR: unique} to hold. As long as we can find the illumination sources $\{ f_i \}_{i=0}^K$ such that the set $\{x\in\Omega\mid \aver{u_i}\neq \aver{u_j}\}$ is dense for each pair $i\neq j$, then the result still holds. However, it is more difficult to show that this holds in the general transport regime, although it seems obvious.


Using Riesz-Thorin interpolation theorem~\cite{stein1971introduction} and H\"older inequality, we can easily derive the following $L^p$ stability estimate for the transport model.
\begin{corollary}
Under the same conditions as Theorem~\ref{thm: pat}, there exists a constant $C > 0$ such that 
\begin{equation*}
  \|  \Sigma_a( \aver{u}) - \widetilde{\Sigma}_a( \aver{\widetilde{u}})\|_{L^p(\Omega)} \le C  |\operatorname{Vol}(\Omega)|^{\frac{1}{p}(1-\frac{1}{p})} \left[ g(\|f_{-}\|_{\infty})\right]^{1-\frac{1}{p}}\cdot \|H - \widetilde{H}\|_{L^p(\Omega)}^{\frac{1}{p}}.
\end{equation*}
\end{corollary}

\section{Transition of stability regimes}\label{Sec: trans}

We study in this section the transition of stability from the transport to the diffusive regimes. We refer interested readers to~\cite{LaLiUh-SIAM19,ZhZh-SIAM19} for related investigations.

Let us consider the diffusion limit process of the transport model: 
\begin{equation}\label{eq: scale rte}
\begin{aligned}
    v\cdot \nabla u +  \Sigma_{a,\eps}(\aver{u}) u &=\Sigma_{s,\eps} (\cK u - u), &\text{ in }&X \\\quad u(x,v)&= f_{-},\quad &\text{ on }&\Gamma_{-},
\end{aligned}
\end{equation}
where $\Sigma_{a,\eps}:= \eps \Sigma_a$ and $\Sigma_{s, \eps}:=\eps^{-1}\Sigma_s$ are the \textit{scaled coefficients}, the parameter $\eps$ is the Knudsen number, which is the ratio between the mean free path and the characteristic size of the domain. Then the constant of the $L^1$ stability estimate in Theorem~\ref{thm: pat} becomes $$\frac{(1-\mu)^{-1}\Sigma_{a,\eps} + \Sigma_{s,\eps}}{\Sigma_{s,\eps} + \Sigma_{a,\eps}}\sim \frac{\Sigma_a}{\Sigma_s}\eps^2 \exp\left(\operatorname{diam}(\Omega) \cdot ( \eps g(\|f_{-}\|_{\infty}) + \frac{1}{\eps}\overline{\Sigma}_s ) \right), $$
which grows exponentially fast to infinity as $\eps\to 0^{+}$, leading to a disappointing discrepancy from the result in the diffusion regime in Theorem~\ref{thm: diff-pat}. This is partially caused by the boundary layer effect, which causes the solution to deviate from its diffusion limit. In the following, we show that the gap can be closed by introducing a damping factor near the boundary under the following assumptions.
 
\begin{enumerate}
    \item [$(\cD)$.] The scattering is isotropic, that is, $\cK f = \aver{f}$.
    \item [$(\cE)$.]  Let $k>\frac{d}{2}+2$. The absorption coefficient $\Sigma_a(\aver{u})$ is a polynomial: 
    \begin{equation*}
    \Sigma_a(\aver{u}) = \sum_{j=0}^{n} \Sigma_{a, j} |\aver{u}|^j,    
    \end{equation*}
    where all coefficients $\Sigma_{a, j}\in C^{k, 1}(\overline{\Omega})$, $j=0,1,\cdots, n$, are strictly positive.
    The scattering coefficient $\Sigma_{s}\in C^{k, 1}(\overline{\Omega})$ is strictly positive and $\partial\Omega\in C^{k+2}$. The illumination function $f_{-}(x, v)= f_0(x)\big|_{\partial\Omega}$, where $f_0\in H^{k+2}(\bbR^d)$ is a positive function. 
    
\end{enumerate}

\subsection{Preliminaries}

We first prove some preliminaries under the assumption that $\Sigma_{a}\in C^{k, 1}(\overline{\Omega})$. Later, we will relax this assumption. Let $\cP_{\eps}$ be the integral operator defined through the relation
\begin{equation*}
\begin{aligned}
    \cP_{\eps} f(x) := \int_{\bbS^{d-1}} \int_0^{\tau_{-}(x,v)}E_{\eps}(x, x-sv) {\Sigma_{t,\eps}(x-sv)} f(x-sv) ds dv,
\end{aligned}
\end{equation*}
where $\Sigma_{t,\eps}:= \Sigma_{a,\eps} + \Sigma_{s,\eps} = \eps\Sigma_a + \eps^{-1}\Sigma_s$ and $E_{\eps}(x, x-sv)$ is defined by
\begin{equation*}
    E_{\eps}(x, x-sv) := \exp\left(-\int_0^s \Sigma_{t,\eps}(x-tv) dt \right).
\end{equation*}
The operator $\cP_{\eps}$ is also known as the Peierls integral operator.
Next, we present a useful result about $\cP_{\eps}$. The calculations can be found in~\cite{ReZhZh-JCP19,ZhZh-SIAM19}. 
\begin{lemma}\label{Lem: Pei}
 $\cP_{\eps}: L^2(\Omega)\to L^2(\Omega)$ is a compact, self-adjoint integral operator. Under a polar coordinate transform,
\begin{equation*}
    \cP_{\eps} f(x) = \frac{1}{\nu_{d}} \int_{\Omega} \frac{E_{\eps}(x, y)}{|x - y|^{d-1}} {\Sigma_{t,\eps}(y)} f(y) dy,
\end{equation*}
where $\nu_k$ denotes the area of the unit sphere in $\bbR^d$.
\end{lemma}
\begin{definition}
    Under the assumption $(\cE)$, we define $\cA_{\eps}$ as the differential operator 
    \begin{equation*}
        \cA_{\eps} f(x) = -\nabla\cdot \left(\frac{1}{\Sigma_{t,\eps}(x)} \nabla f(x)\right). 
    \end{equation*}
    Let $\lambda_{\eps} > 0$ be the first Dirichlet eigenvalue of $\cA_{\eps}$. Then $\lambda_{\eps}$ is simple and $\lambda_{\eps} = \Theta(\eps)$ in the sense of big $\Theta$ notation, that is, there exists constants $c_1, c_2 > 0$ independent of $\eps$ that $c_1 \eps <\lambda_{\eps} < c_2\eps $. We denote by $\Phi_{\eps}$ the corresponding eigenfunction, which is normalized and is strictly positive in the interior of $\Omega$.
\end{definition}
Then, we have $\Phi_{\eps} \in H^{k+2}(\Omega)$, see~\cite[Theorem 8.13]{gilbarg_elliptic_2000}, and by Sobolev embedding, we can have $\Phi_{\eps}\in C^4(\Omega)$ with a uniform bound $\|\Phi_{\eps}\|_{C^4(\Omega)} < C < \infty$ for all $\eps > 0$. 

\begin{lemma}\label{LEM: Limit}
As $\eps \to 0$, the eigenvalue satisfies $\eps^{-1}\lambda_{\eps}\to \lambda^{\ast}$, where $\lambda^{\ast}$ corresponds to the first eigenvalue of the operator 
\begin{equation}\label{EQ: A STAR}
    \cA^{\ast} f = -\nabla \cdot \left(\frac{1}{\Sigma_s(x)} \nabla f(x) \right),
\end{equation}
The eigenfunction $\Phi_{\eps}$ converges to the first eigenfunction of $\cA^{\ast}$.
\end{lemma}
\begin{proof}
    We observe that $$\frac{\lambda_{\eps}}{\eps} = \inf_{u\in H^1_0(\Omega)} \frac{\int_{\Omega} (\Sigma_{s}(x) + \eps^2\Sigma_{a}(x))^{-1} |\nabla u|^2 dx }{\int_{\Omega} |u|^2 dx }$$
    must monotonically decrease as $\eps \to 0$, thus the limit $\lambda^{\ast} > 0$ exists. Taking any sequence of the eigenfunctions $\Phi_{\eps}$ as $\eps \to 0$, we have a converging subsequence to $\Phi^{\dagger}\in C^{2,\alpha}(\Omega)$, which leads to 
    \begin{equation*}
        \cA^{\ast} \Phi^{\dagger} = \lambda^{\ast} \Phi^{\dagger}.
    \end{equation*}
    Thus $\Phi^{\dagger}$ is the first eigenfunction of $\cA^{\ast}$.
\end{proof}

\begin{lemma}\label{LEM: INN}
If $\operatorname{dist}(x, \partial\Omega) \ge \frac{\kappa |\log\eps|}{\inf_{\Omega}\Sigma_{t,\eps}}$, then  
    \begin{equation*}
        \cP_{\eps} \Phi_{\eps} (x)= \Phi_{\eps}(x) - \frac{1}{d\Sigma_{t,\eps}(x)} \lambda_{\eps} \Phi_{\eps}(x) + \cO(\eps^{\min(\kappa, 4)}).
    \end{equation*}
\end{lemma}
\begin{proof}
Let $r = \frac{\kappa |\log\eps|}{\inf_{\Omega}\Sigma_{t,\eps}}$ with $\kappa = \cO(1)$. Thus $r = \cO(\eps|\log\eps|)$. For any point $x$ such that $\operatorname{dist}(x, \partial\Omega) \ge r$, the ball $B_r(x) \subset \Omega$. Using Lemma~\ref{Lem: Pei}, we get
\begin{equation*}
\begin{aligned}
      \cP_{\eps} \Phi_{\eps} (x) &= \frac{1}{\nu_{d}}\int_{B_{r}(x)} \frac{E_{\eps}(x, y)}{|x - y|^{d-1}} \Sigma_{t,\eps}(y) \Phi_{\eps}(y) dy + \frac{1}{\nu_{d}}\int_{|y - x|>r} \frac{E_{\eps}(x, y)}{|x - y|^{d-1}} \Sigma_{t,\eps}(y) \Phi_{\eps}(y) dy.
\end{aligned}
\end{equation*}
The integral outside $B_{r}(x)$ is bounded by 
\begin{equation}\label{eq: out}
     \frac{\sup_{\Omega}\Sigma_{t,\eps}}{\inf_{\Omega}\Sigma_{t,\eps}} \|\Phi_{\eps}\|_{\infty} \exp(-r \inf_{\Omega}\Sigma_{t,\eps}) = \frac{\sup_{\Omega}\Sigma_{t,\eps}}{\inf_{\Omega}\Sigma_{t,\eps}} \|\Phi_{\eps}\|_{\infty} \eps^{\kappa} = \cO(\eps^{\kappa}).
\end{equation}
For the integral in $B_{r}(x)$, we expand both $E_{\eps}(x, y)$ and $U_{\eps} = \Sigma_{t,\eps}\Phi_{\eps}$ into Taylor series. Let $t := |x - y|$ and $v = \frac{x - y}{|x - y|}$, 
\begin{equation*}
\begin{aligned}
    E_{\eps}(x, y) &= e^{-t \Sigma_{t,\eps}(x)}\exp\left(- \frac{t^2}{2} \nabla\Sigma_{t,\eps}(x) \colon v^{\otimes 1}  - \frac{t^3}{6} \nabla^2\Sigma_{t,\eps}(x) \colon v^{\otimes 2}   + \cO(\|\Sigma_{t,\eps}\|_{C^4})t^4\right)\\ 
    &= e^{-t \Sigma_{t,\eps}(x)} \left[1  - \frac{t^2}{2} \nabla\Sigma_{t,\eps}(x) \colon v^{\otimes 1}  - \frac{t^3}{6} \nabla^2\Sigma_{t,\eps}(x) \colon v^{\otimes 2}   + \frac{t^4}{8} (\nabla \Sigma_{t,\eps}(x)\colon v^{\otimes 1})^2 \right. \\
    &\qquad \left. + \cO((\|\Sigma_{t,\eps}\|_{C^4} t^4 + \|\Sigma_{t,\eps}\|_{C^4}^2 t^5 ) \right. \Big],\\
    U_{\eps}(y) &= U_{\eps}(x) + t \nabla U_{\eps}(x)\colon v^{\otimes 1} + \frac{t^2}{2} \nabla^2 U_{\eps}(x) \colon v^{\otimes 2} + \frac{t^3}{6} \nabla^3 U_{\eps}(x) \colon v^{\otimes 3} + \cO(\|U_{\eps}\|_{C^4}) t^4.
\end{aligned}
\end{equation*}
Therefore, 
\begin{equation*}
\frac{1}{\nu_{d}}\int_{B_{r}(x)} \frac{E_{\eps}(x, y)}{|x - y|^{d-1}} U_{\eps}(y) dy  = \cI_1 + \cI_2 + \cI_3 + \cO(\eps^4),
\end{equation*}
where $\cI_1$ represents the diffusion approximation part in homogeneous media,
\begin{equation*}
\begin{aligned}
    \cI_1 &= \int_{\bbS^{d-1}} \int_0^r e^{-r \Sigma_{t,\eps}(x)}\left(U_{\eps}(x) + \frac{t^2}{2} \nabla^2 U(x)\colon v^{\otimes 2}\right) dt dv \\
    &= \frac{1 - e^{-r \Sigma_{t,\eps}(x)}}{\Sigma_{t,\eps}(x)} U_{\eps}(x) +  \Delta U_{\eps} \frac{1 - (1 + r \Sigma_{t,\eps} + \frac{1}{2}r^2 \Sigma_{t,\eps}^2 )e^{-r \Sigma_{t,\eps}(x)}}{d \Sigma_{t,\eps}^3}.
\end{aligned}
\end{equation*}
The second term $\cI_2$ represents the contribution from 1st order heterogeneity, 
\begin{equation*}
\begin{aligned}
    \cI_2 &=  \int_{\bbS^{d-1}} \int_0^r e^{-r \Sigma_{t,\eps}(x)} \left(-\frac{t^3}{2} (\nabla\Sigma_{t,\eps}(x) \cdot v) (\nabla U_{\eps}(x)\cdot v) + \frac{t^4}{8} U_{\eps}(x) (\nabla \Sigma_{t,\eps}(x)\cdot v)^2 \right) \\
    &= (-3 \nabla\Sigma_{t,\eps}(x)\cdot \nabla U_{\eps}(x)) \frac{1 - (1 + r\Sigma_{t,\eps} + \frac{1}{2}r^2 \Sigma_{t,\eps}^2 + \frac{1}{6} r^3 \Sigma_{t,\eps}^3) e^{-r \Sigma_{t,\eps}(x)} }{d \Sigma_{t,\eps}(x)^4} \\
    &\qquad + 3 |\nabla\Sigma_{t,\eps}(x)|^2 U_{\eps}(x) \frac{1 - (1 + r\Sigma_{t,\eps} + \frac{1}{2}r^2 \Sigma_{t,\eps}^2 + \frac{1}{6} r^3 \Sigma_{t,\eps}^3 + \frac{1}{24}r^4 \Sigma_{t,\eps}^4) e^{-r \Sigma_{t,\eps}(x)} }{d\Sigma_{t,\eps}^5(x)}. 
\end{aligned}
\end{equation*}
The third term $\cI_3$ comes from 2nd order heterogeneity,
\begin{equation*}
\begin{aligned}
    \cI_3 &= \int_{\bbS^{d-1}} \int_0^r  e^{-r \Sigma_{t,\eps}(x)} \frac{t^3}{6} U_{\eps}(x) \nabla^2 \Sigma_{t,\eps}(x) \colon v^{\otimes 2} dt dv \\
    &= -U_{\eps}(x) \Delta \Sigma_{t,\eps}(x) \frac{1 - (1 + r\Sigma_{t,\eps} + \frac{1}{2}r^2 \Sigma_{t,\eps}^2 + \frac{1}{6} r^3 \Sigma_{t,\eps}^3) e^{-r \Sigma_{t,\eps}(x)} }{d \Sigma_{t,\eps}^4(x)}. 
\end{aligned}
\end{equation*}
The terms in the form of 
\begin{equation*}
   \frac{  (1 + r\Sigma_{t,\eps} + \frac{1}{2}r^2 \Sigma_{t,\eps}^2 + \cdots + \frac{1}{(s-1)!} r^{s-1} \Sigma_{t,\eps}^{s-1})  e^{-r \Sigma_{t,\eps}(x)}}{\Sigma_{t,\eps}(x)^{s}}
\end{equation*}
are bounded by $\cO(|\log\eps|^{s-1} \eps^{s+\kappa})$. Then use $U_{\eps} = \Sigma_{t,\eps}\Phi_{\eps}$, such terms are $\cO(\eps^{\kappa})$ in $\cI_1$, $\cO(|\log\eps|^4\eps^{2+\kappa})$ in $\cI_2$, and $\cO(|\log\eps|^3 \eps^{2+\kappa})$ in $\cI_3$.  
After dropping these terms and simplification, we obtain
\begin{equation*}
\begin{aligned}
\frac{1}{\nu_{d}}\int_{B_{r}(x)} \frac{E_{\eps}(x, y)}{|x - y|^{d-1}} U_{\eps}(y) dy &= \Phi_{\eps}(x) + \frac{1}{d\Sigma_{t,\eps}} \cA_{\eps} \Phi_{\eps} + \cO( \eps^4 + \eps^{\kappa}) \\
&= \Phi_{\eps}(x) - \frac{1}{d\Sigma_{t,\eps}} \lambda_{\eps} \Phi_{\eps} + \cO(  \eps^4 + \eps^{\kappa}). 
\end{aligned}
\end{equation*}
Combining this with~\eqref{eq: out}, we obtain the desired estimate.
\end{proof}

\begin{lemma}\label{LEM: OUT}
Let $\ell = \operatorname{dist}(x, \partial\Omega)$. When $\ell < \frac{8|\log\eps|}{\inf_{\Omega}\Sigma_{t,\eps}}$, then $\Phi_{\eps}(x) = \Theta(\ell) $. 
\end{lemma}
\begin{proof}
    By Hopf Lemma, $\bn\cdot \nabla \Phi^{\dagger} < 0$ on boundary. Since $\partial\Omega\in C^{k+2}$ is a compact set and $\nabla \Phi^{\dagger}$ is continuous, there is an absolute constant $c > 0$ that $\bn\cdot \nabla \Phi^{\dagger} < -c$.  Then use Lemma~\ref{LEM: Limit}, we have $\Phi_{\eps}\to \Phi^{\dagger}$ in $C^{2,\alpha}(\Omega)$. Therefore, $\Phi_{\eps}(x) = \Theta(\ell)$ for sufficiently small $\ell$.
\end{proof}

\begin{lemma}\label{Lem: boundary}
Let $\ell = \operatorname{dist}(x, \partial\Omega)$. When $\ell < \frac{4|\log\eps|}{\inf_{\Omega}\Sigma_{t,\eps}}$, then 
\begin{equation*}
    \cP_{\eps} \Phi_{\eps}(x) = \cO(\eps|\log\eps|).
\end{equation*}
\end{lemma}
\begin{proof}
    Let $T = B_{r}(x) \cap \Omega$, where $r = \frac{4|\log\eps|}{\inf_{\Omega}\Sigma_{t,\eps}}$. Outside this region, the integral of $\cP_{\eps}\Phi_{\eps}$ is bounded by $\cO(\eps^4)$ according to~\eqref{eq: out}. By Lemma~\ref{LEM: OUT} we have
    \begin{equation*}
        \operatorname{dist}(x, \partial\Omega) \le \ell + r = \cO(\eps|\log\eps|),\quad x\in T.
    \end{equation*}
    Therefore, 
    \begin{equation*}
        \cP_{\eps}\Phi_{\eps}(x) = \cO(\eps^4) + ( 1 - \exp(- r \inf_{\Omega}\Sigma_{t,\eps}) ) \cO(\eps|\log\eps|) = \cO(\eps|\log\eps|).
    \end{equation*}
    The result is proven.
\end{proof}
The following theorem characterizes the principal eigenvalue for $\cP_{\eps}$. For slab geometry with homogeneous coefficients, all of the eigenvalues of $\cP_{\eps}$ can be sharply estimated through careful calculations with the min-max principle, see~\cite{lobarev1989eigenvalues}. 

\begin{theorem}\label{Thm: eigen}
The principal eigenvalue of $\cP_{\eps}$ is the spectral radius $\rho(\cP_{\eps})$ and satisfies  
$$c_1 \eps^2 < 1 - \rho(\cP_{\eps}) < c_2 \eps^2$$ for certain constants $c_1, c_2 > 0$.
\end{theorem}

\begin{proof}
Since $\cP_{\eps}$ is a positive operator, the classical Krein-Rutman theorem~\cite{krein1948linear} implies that the spectral radius $\lambda:= \rho(\cP_{\eps})$ is the principal eigenvalue of multiplicity one, and the associated eigenfunction is strictly positive. To apply the Courant-Fischer-Weyl min-max principle~\cite{reed1978iv}, we have to symmetrize the operator. Define $\cP_{\eps, \operatorname{sym}}:= \Sigma_{t,\eps}^{1/2} \cP_{\eps} \Sigma_{t,\eps}^{-1/2}$, then the eigenvalues of $\cP_{\eps}$ and $\cP_{\eps, \mathrm{sym}}$ are the same. Now, we can apply the min-max principle, the principal eigenvalue of $\cP_{\eps}$ is bounded below by:
\begin{equation*}
\begin{aligned}
    \rho(\cP_{\eps}) &= \max_{f\in L^2(\Omega)} \frac{\aver{\cP_{\eps, \mathrm{sym}} f, f}}{\aver{f, f}} \ge \frac{\aver{\cP_{\eps, \mathrm{sym}} \Sigma_{t,\eps}^{1/2}\Phi_{\eps}, \Sigma_{t,\eps}^{1/2}\Phi_{\eps}}}{\aver{\Sigma_{t,\eps}^{1/2}\Phi_{\eps}, \Sigma_{t,\eps}^{1/2}\Phi_{\eps}}}\\
    &= \left.\dint_{\Omega} \Sigma_{t,\eps}(x) \Phi_{\eps}(x) \cdot \cP_{\eps} \Phi_{\eps}(x) dx \middle/ \dint_{\Omega} \Sigma_{t,\eps} (x) \Phi_{\eps}^2(x) dx\right. .
\end{aligned} 
\end{equation*}
To estimate the numerator, we decompose the domain into $\Omega = \Omega_r + \Omega_r^{\complement}$, where $\Omega_r$ refers to the interior region $\Omega_r := \{x\in\Omega\mid \operatorname{dist}(x, \partial\Omega) \ge r := \frac{4|\log\eps|}{\inf_{\Omega}\Sigma_{t,\eps}}\}$. By Lemma~\ref{LEM: INN}, we have 
\begin{equation*}
    \cP_{\eps}\Phi_{\eps}(x) = (1 - \frac{\lambda_{\eps}}{d\Sigma_{t,\eps}} )\Phi_{\eps}(x) + \cO(\eps^4),\quad x\in \Omega_r.
\end{equation*}
Using Lemma~\ref{LEM: OUT}, we conclude $\Phi_{\eps}(x)$ is bounded below by $\cO(r)$ on $\Omega_r$, thus $\Phi_{\eps}(1 - \frac{\lambda_{\eps}}{d\Sigma_{t,\eps}}) \ge C  \eps^3|\log\eps| $ for a certain $C > 0$, we can safely absorb $\cO(\eps^4)$ into it. Therefore, with the estimate in Lemma~\ref{Lem: boundary},
\begin{equation*}
\begin{aligned}
\dint_{\Omega} \Sigma_{t,\eps}(x) \Phi_{\eps}(x) \cdot \cP_{\eps} \Phi_{\eps}(x) dx &\ge (1 - c'\eps^2) \int_{\Omega_r} \Sigma_{t,\eps}(x)  |\Phi_{\eps}(x)|^2 dx \\&\quad + C\int_{\Omega_r^{\complement}} \Sigma_{t,\eps}(x) \operatorname{dist}(x,\partial\Omega)\cdot (\eps|\log\eps|) dx \\
&=(1 - c'\eps^2) \int_{\Omega} \Sigma_{t,\eps}(x)  |\Phi_{\eps}(x)|^2 dx + \cO(\eps^2 |\log\eps|^3) \\
&= (1 - c'\eps^2 + \cO(\eps^3|\log\eps|^3)) \int_{\Omega} \Sigma_{t,\eps}(x)  |\Phi_{\eps}(x)|^2 dx.
\end{aligned}
\end{equation*}
Hence, there exists $c_1$ that $\rho(\cP_{\eps}) > 1 - c_1 \eps^2$. 

Next, we prove the upper bound of $\rho(\cP_{\eps})$. By Whitney's extension theorem, we can extend the coefficient $\Sigma_{s}, \Sigma_a\in C^4(\bbR^d)$. Then, we take a convex neighborhood set $\Omega^{\diamond}\supset \Omega$ that 
\begin{equation*}
    \inf_{\Omega^{\diamond}} \Sigma_{a,\eps} > \frac{1}{2} \inf_{\Omega} \Sigma_{a,\eps}\quad     \inf_{\Omega^{\diamond}} \Sigma_{s,\eps} > \frac{1}{2} \inf_{\Omega} \Sigma_{s,\eps}.
\end{equation*}
Then $\operatorname{dist}(\Omega, \partial\Omega^{\diamond}) = \Theta(1)$ due to boundedness of derivatives.
Let $r^{\diamond} = \frac{4|\log\eps|}{\inf_{\Omega^{\diamond}\Sigma_{t,\eps}}} = o(1)$. Then $\Omega\subset \{x\in \Omega^{\diamond} \mid \operatorname{dist}(x, \partial\Omega^{\diamond}) \ge r^{\diamond} \}$ for sufficiently small $\eps$. Denote $\cP_{\eps}^{\diamond}$ (resp. $\cA_{\eps}^{\diamond}$) the extension of $\cP_{\eps}$ (resp. $\cA_{\eps}$). Let $( \lambda_{\eps}^{\diamond}, \Phi_{\eps}^{\diamond} )$ be the first eigen-pair of $\cA_{\eps}^{\diamond}$, then use Lemma~\ref{LEM: INN} on domain $\Omega^{\diamond}$, 
\begin{equation*}
\begin{aligned}
    \cP_{\eps}^{\diamond} \Phi_{\eps}^{\diamond} = \Phi_{\eps}^{\diamond} - \frac{1}{d\Sigma_{t,\eps}} \lambda_{\eps}^{\diamond} \Phi_{\eps}^{\diamond} + \cO(\eps^4), \quad x\in\Omega.
\end{aligned}
\end{equation*}
Therefore, we have 
\begin{equation*}
\begin{aligned}
    \cP_{\eps} ( \Phi_{\eps}^{\diamond}\chi_{\Omega} ) &< \cP_{\eps}^{\diamond} \Phi_{\eps}^{\diamond} =  \Phi_{\eps}^{\diamond} - \frac{1}{d\Sigma_{t,\eps}} \lambda_{\eps}^{\diamond} \Phi_{\eps}^{\diamond} + \cO(\eps^4) \\ & <  \Phi_{\eps}^{\diamond} (1 - c''\eps^2 ),\quad x\in\Omega. 
\end{aligned}
\end{equation*}
It implies that $\rho(\cP_{\eps}) < 1 - c''\eps^2 $ by Gelfand's formula.
\end{proof}

The above theorem is still valid if the absorption coefficient $\Sigma_{a}$ has a rough perturbation in order $\cO(\eps^{\beta})$, $\beta > 0$.
\begin{corollary}\label{cor: eigen}
    Let $\Sigma_{a}^w\in L^{\infty}(\Omega)$ be such that  
    \begin{equation*}
        \|\Sigma_{a} - \Sigma_a^w\|_{L^{\infty}(\Omega)} < C \eps^{\beta},\quad \beta > 0.
    \end{equation*}
    Then the bound of Theorem~\ref{Thm: eigen} is still true when $\Sigma_{a}$ is replaced with $\Sigma_{a}^w$.
\end{corollary}
\begin{proof}
    We define two functions $\Sigma_{a}^{\uparrow} := \Sigma_{a} + C\eps^{\beta}$ and $\Sigma_{a}^{\downarrow} := \Sigma_{a} - C\eps^{\beta}$. Then $\Sigma_{a}^{\uparrow}, \Sigma_{a}^{\downarrow}\in C^{k, 1}(\overline{\Omega})$. Let $\cP_{\eps}^{\uparrow}$ and $\cP_{\eps}^{\downarrow}$ be the Peirels integral operators for $\Sigma_{a}^{\uparrow}$ and $\Sigma_{a}^{\downarrow}$, respectively.
    Then 
    \begin{equation*}
        \rho(\cP_{\eps}^{w}) < \rho\left(\cP_{\eps}^{\downarrow}\right) \sup\frac{\eps^2 \Sigma_{a}^{\uparrow} + \Sigma_s}{\eps^2 \Sigma_{a}^{\downarrow} + \Sigma_s} < (1 - \Theta(\eps^2) ) (1 + \cO(\eps^{2+\beta})) = 1 - \Theta(\eps^2).
    \end{equation*}
    Similarly, 
    \begin{equation*}
        \rho(\cP_{\eps}^{w}) > \rho\left(\cP_{\eps}^{\uparrow}\right) \sup\frac{\eps^2 \Sigma_{a}^{\downarrow} + \Sigma_s}{\eps^2 \Sigma_{a}^{\uparrow} + \Sigma_s} > (1 - \Theta(\eps^2)) (1 - \cO(\eps^{2+\beta})) = 1 - \Theta(\eps^2).
    \end{equation*}
    Thus, the conclusion of Theorem~\ref{Thm: eigen} does not change.
\end{proof}

\subsection{Main estimate}

Under the conditions $(\cD)$ and $(\cE)$, with a slight modification of the classical diffusion approximation theory~\cite[Chapter XXI, \S 5.2, Theorem 2]{dautray2012mathematical_v6} and standard fixed-point theory, we can show $|\widetilde{u}(x, v) - \widetilde{U}(x)| = \cO(\eps)$ if $\eps$ is sufficiently small, where $\widetilde{U}$ is the positive solution to the semilinear diffusion equation 
\begin{equation}\label{eq: nonlinear diff}
\begin{aligned}
    \nabla \left(\frac{1}{d\Sigma_{s}}\nabla \widetilde{U}\right) - \widetilde{\Sigma}_{a}(|\widetilde{U}|) \widetilde{U} &= 0,\quad &\text{ in }&\Omega, \\\widetilde{U}(x) &=f_0, \quad &\text{ on }&\partial\Omega.
\end{aligned}
\end{equation}
However, to make the diffusion approximation valid, we need the solution $\widetilde{U}\in C^{3,\alpha}(\overline{\Omega})$. 
\begin{theorem}\label{Thm: reg}
    Suppose $\widetilde{\Sigma}_a$ and $\Sigma_s$ satisfy condition $(\cE)$, then ~\eqref{eq: nonlinear diff} admits a unique solution in $C^{4}(\overline{\Omega})$.
\end{theorem}
\begin{proof}
    The uniqueness of the solution $\widetilde{U}\in H^1(\Omega)$ is guaranteed by the usual variational method. To show that $\widetilde{U}$ has higher regularity, we use the idea in~\cite{taylor1996partial} and consider the modified equation 
\begin{equation}\label{eq: nonlinear diff map}
\begin{aligned}
    \nabla \left(\frac{1}{d\Sigma_{s}}\nabla \widetilde{W}\right) - \widetilde{\Sigma}_{a}(\widetilde{W}) \chi(\widetilde{W}) \widetilde{W} &= 0,\quad &\text{ in }&\Omega, \\\widetilde{W}(x) &=f_0, \quad &\text{ on }&\partial\Omega.
\end{aligned}
\end{equation}
where $\chi$ is a $C^{\infty}$ cutoff function that $\chi(x)=1$ if $0 < \theta \le x \le \sup_{\partial\Omega} f_0$, and $\chi(x) = 0$ for $x < 0$ or $x > 2 \sup_{\partial\Omega} f_0$. The parameter $\theta = \inf_{\Omega} \widetilde{V} > 0$ is the minimum of the solution $\widetilde{V}\in C^{4}(\overline{\Omega})$ to the following linear elliptic equation:
\begin{equation}\label{eq: nonlinear diff map 2}
\begin{aligned}
    \nabla \left(\frac{1}{d\Sigma_{s}}\nabla \widetilde{V}\right) - \widetilde{\Sigma}_{a}(\sup_{\partial\Omega} f_0) \widetilde{V} &= 0,\quad &\text{ in }&\Omega, \\\widetilde{V}(x) &=f_0, \quad &\text{ on }&\partial\Omega.
\end{aligned}
\end{equation}
Then, the modified equation~\eqref{eq: nonlinear diff map} admits a unique solution $\widetilde{W}\in H^1(\Omega)$, and we have $\widetilde{\Sigma}_{a}(\widetilde{W}) \chi(\widetilde{W}) \widetilde{W}\in H^1(\Omega)$, which implies $\widetilde{W}\in H^2(\Omega)$ since the boundary condition is sufficiently regular. This inductively pumps the regularity up till $\widetilde{W}\in H^{k+2}(\Omega)$. Hence $\widetilde{W}\in C^{4}(\overline{\Omega})$ by Sobolev embedding. Then, by comparison principle, $\theta \le \widetilde{W} \le \sup_{\partial \Omega} f_0$. This makes $\widetilde{W}$ the solution to~\eqref{eq: nonlinear diff}.
\end{proof}
With the diffusion approximation error $|\widetilde{u}(x, v) - \widetilde{U}(x)| = \cO(\eps)$ and $\widetilde{U} > 0$ is uniformly bounded from below, we can find the estimate $\widetilde{u}(x, v)/ \aver{\widetilde u}(x) \le 1 + \widetilde{C} \eps$ for a certain $\widetilde{C} > 0$. The following theorem closes the gap in stability estimates between the diffusion and transport regimes under certain conditions.

\begin{theorem}
  Under the conditions $(\cA)$, $(\cD)$, and $(\cE)$, assume that $H_{\eps}$ and $\widetilde{H}_{\eps}$ are data associated to the scaled $\Sigma_{a,\eps}$ and $\widetilde{\Sigma}_{a,\eps}$, respectively. If $H_{\eps} = \widetilde{H}_{\eps}$, then $\Sigma_{a,\eps}( \aver{u}) = \widetilde{\Sigma}_{a,\eps}( \aver{\widetilde{u}})$. Moreover, there exists a constant $C > 0$ independent of $\eps$ that
  \begin{equation*}
     \left\| \Psi_{\eps} \frac{\Sigma_{a,\eps} (\aver{u})- \widetilde{\Sigma}_a(\aver{\widetilde{u}})}{\Sigma_{a,\eps}( \aver{u})} \aver{\widetilde{u}}\right\|_{L^1(\Omega)} \le C \left\|\Psi_{\eps}\frac{H_{\eps} - \widetilde{H}_{\eps}}{\Sigma_{a,\eps}( \aver{u})}\right\|_{L^1(\Omega)},
\end{equation*}
where $(\mu_{\eps}, \Psi_{\eps})$ is the leading eigenpair of $\cP_{\eps}$ with nonlinear absorption $\Sigma_{a}=\Sigma_a(\aver{u})$.
\end{theorem}
\begin{proof}
Recall that the perturbation $\delta u:=u-\widetilde u$ satisfies the transport equation
\begin{equation*}
  \begin{aligned}
         v\cdot \nabla \delta u + \Sigma_{t,\eps}(x) \delta u - \Sigma_{s,\eps}(x)\aver{\delta u} &= - \delta\Sigma_{a,\eps} \widetilde{u} \quad &&\text{ in }X, \\
         \delta u &= 0\quad &&\text{ on }\Gamma_{-}.
  \end{aligned}
\end{equation*}
Moreover, from the data, we have $H_{\eps} - \widetilde{H}_{\eps} = \Sigma_a( \aver{u}) \aver{\delta u} + \delta\Sigma_a \aver{\widetilde{u}}$. First, we have the following lower bound, which is similar to~\eqref{eq: less} in Theorem~\ref{thm: pat}:
\begin{equation}\label{eq: less 2}
     \Sigma_{t,\eps}|\aver{\delta u}| \ge \mu_{\eps} {\Sigma_{s,\eps}}|\aver{\delta u}| + \left( {\Sigma_{a,\eps} + (1 - \mu_{\eps})\Sigma_{s,\eps}}\right) \left( \left|\frac{\delta\Sigma_{a,\eps}}{\Sigma_{a,\eps}( \aver{u})} \aver{\widetilde{u}}\right| - \left|\frac{H_{\eps} - \widetilde{H}_{\eps}}{\Sigma_{a,\eps}( \aver{u})}\right| \right).
\end{equation}
For the upper bound, we can use $|\aver{\delta u}|$ instead of $\aver{|\delta u|}$ due to the isotropic scattering assumption we made. This leads to
\begin{equation*}
\begin{aligned}
  |\aver{\delta u}(x)| &= \left| \int_{\bbS^{d-1}} \delta u(x, v) dv \right| \\ &=  \left| \int_{\bbS^{d-1}}  \int_0^{\tau_{-}(x, v)} E(x, x-sv)   \left(- \delta\Sigma_a(x-sv) \widetilde{u}(x-sv, v) + \Sigma_s \aver{\delta u} (x-sv)\right) ds  dv  \right| \\
  &\le \left| \int_{\bbS^{d-1}} \int_0^{\tau_{-}(x, v)} E(x, x-sv)   \delta\Sigma_a(x-sv) \widetilde{u}(x-sv, v) ds dv \right| \\& \quad + \left| \int_{\bbS^{d-1}} \int_0^{\tau_{-}(x, v)} E(x, x-sv)  \Sigma_s \aver{\delta u} (x-sv) ds dv \right|.
\end{aligned}
\end{equation*}
Multiply the above inequality with $\Sigma_{t,\eps} \Psi_{\eps}$ and integrate, 
\begin{equation*}
 \| \Psi_{\eps} \Sigma_{t,\eps} \aver{\delta u} \|_{L^1(\Omega)} \le M_1 + M_2.
\end{equation*}
where $M_1$ and $M_2$ are defined as follows.
\begin{equation*}
\begin{aligned}
    M_1 &= \left| \int_{\Omega} \Sigma_{t,\eps}(x) \Psi_{\eps}(x) \int_{\bbS^{d-1}}\int_0^{\tau_{-}(x,v)} E_{\eps}(x, x-sv) \delta \Sigma_{a,\eps} (x-sv) \widetilde{u}(x-sv, v) ds dv dx \right| \\
    M_2 &=  \left| \int_{\Omega} \Sigma_{t,\eps}(x) \Psi_{\eps}(x) \cP_{\eps} \frac{\Sigma_{s,\eps}}{\Sigma_{t,\eps}}\aver{\delta u}(x) dx \right|.
\end{aligned}
\end{equation*}
Applying the Fubini theorem to $M_2$, we find that
\begin{equation*}
\begin{aligned}
M_2 &= \left| \int_{\Omega} \int_{\Omega} \Sigma_{t,\eps}(x) \Psi_{\eps}(x) \frac{E_{\eps}(x, y)}{|x - y|^{d-1}} \Sigma_{s,\eps}(y) \aver{\delta u}(y) dx d y \right| \\&=  \left| \int_{\Omega}  \Sigma_{s,\eps}(x)\aver{\delta u}(x) \int_{\Omega} \frac{E_{\eps}(x, y)}{|x - y|^{d-1}} \Sigma_{t,\eps}(y) \Psi_{\eps}(y) dy dx \right|  \\&= \mu_{\eps} \left| \int_{\Omega} \Psi_{\eps}(x)   \Sigma_{s,\eps}(x)\aver{\delta u}(x) dx\right| \le \mu_{\eps}  \int_{\Omega} \Psi_{\eps}(x)   \Sigma_{s,\eps}(x)|\aver{\delta u}(x)| dx.
\end{aligned}
\end{equation*}
Note that this upper bound of $M_2$ matches the first term in the lower bound~\eqref{eq: less 2} after multiplying $\Psi_{\eps}$.
To estimate $M_1$, we have the following upper bound:
\begin{equation*}
\begin{aligned}
    M_1 &\le (1 + \widetilde{C}\eps) \int_{\Omega} \Sigma_{t,\eps}(x) \Psi_{\eps}(x) \int_{\bbS^{d-1}}\int_0^{\tau_{-}(x,v)} E_{\eps}(x, x-sv) |\delta \Sigma_{a,\eps} (x-sv)| \aver{\widetilde{u}(x-sv)} ds dv dx\\
    &= (1 + \widetilde{C}\eps) \int_{\Omega} \int_{\Omega} \Sigma_{t,\eps}(x) \Psi_{\eps}(x) \frac{E_{\eps}(x, y)}{|x - y|^{d-1}} |\delta\Sigma_{a,\eps}(y)| \aver{\widetilde u}(y) dx d y \\
    &= (1 + \widetilde{C}\eps) \mu_{\eps} \int_{\Omega} \Psi_{\eps}(x) |\delta\Sigma_{a,\eps}(x)| \aver{\widetilde u}(x) dx,
\end{aligned} 
\end{equation*}
where we used the Fubini theorem and $(\mu_{\eps}, \Psi_{\eps})$ is an eigenpair for $\cP_{\eps}$.
Combining this with the lower bound~\eqref{eq: less 2} multiplied by $\Psi_{\eps}$, we can cancel the first term in the lower bound with $M_2$, then obtain 
\begin{equation}\label{EQ: KEY INEQ}
\begin{aligned}
&\int_{\Omega} \Psi_{\eps}(x) \left( {\Sigma_{a,\eps} (1 - (1 + \widetilde{C}\eps) \mu_{\eps} ) + (1 - \mu_{\eps})\Sigma_{s,\eps}}\right) \left|\frac{\delta\Sigma_{a,\eps}}{\Sigma_{a,\eps}( \aver{u})} \aver{\widetilde{u}}\right| dx  \\& \le     \int_{\Omega} \Psi_{\eps}(x) \left( {\Sigma_{a,\eps} + (1 - \mu_{\eps})\Sigma_{s,\eps}}\right)  \left|\frac{H_{\eps} - \widetilde{H}_{\eps}}{\Sigma_{a,\eps}( \aver{u})}\right| dx.
\end{aligned}
\end{equation}
According to Theorem~\ref{Thm: reg}, the diffusion approximation for $u(x, v)$ satisfies $U\in C^4(\overline{\Omega})$, and $|\Sigma_a(U) - \Sigma_a(\aver{u})|=\cO(\eps)$ due to boundedness. By the Corollary~\ref{cor: eigen}, the eigenvalue $\mu_{\eps} = 1 - \Theta(\eps^2)$. We find that the scaling on the left-hand side of~\eqref{EQ: KEY INEQ} is (in big $\Theta$ notation)
\begin{equation*}
     {\Sigma_{a,\eps} (1 - (1 + \widetilde{C}\eps) \mu_{\eps} ) + (1 - \mu_{\eps})\Sigma_{s,\eps}} = - \Theta(\eps) \Sigma_{a,\eps} + \Theta(\eps^2) \Sigma_{s,\eps}  = \Theta(\eps).
\end{equation*}
On the right-hand side of~\eqref{EQ: KEY INEQ}, the scaling is 
\begin{equation*}
    {\Sigma_{a,\eps} + (1 - \mu_{\eps})\Sigma_{s,\eps}} =\Theta(\eps)\Sigma_{a} + \Theta(\eps) \Sigma_{s} = \Theta(\eps).
\end{equation*}
Therefore, the scaling factors cancel, and the stability becomes 
\begin{equation*}
     \left\| \Psi_{\eps} \frac{\delta\Sigma_{a,\eps}}{\Sigma_{a,\eps}( \aver{u})} \aver{\widetilde{u}}\right\|_{L^1(\Omega)} \le C \left\|\Psi_{\eps}\frac{H_{\eps} - \widetilde{H}_{\eps}}{\Sigma_{a,\eps}( \aver{u})}\right\|_{L^1(\Omega)},
\end{equation*}
where $C$ is independent of $\eps$.
\end{proof}

\begin{remark}
The weight function $\Psi_{\eps}$ is small near the boundary, which penalizes the deviation of the solution from its diffusion approximation in the boundary layer. If $\delta\Sigma_{a,\eps} = \widetilde{\Sigma}_{a,\eps}(\aver{\widetilde{u}}) -  \Sigma_{a,\eps}(\aver{u})$ and $H_{\eps} - \widetilde{H}_{\eps}$ are both supported at $\Theta(1)$ distance from the boundary,
then the unweighted $L^1$ stability still applies.
\end{remark}


\section{Concluding remarks}

This work studies forward and inverse problems of a semilinear radiative transport equation with a nonlinear absorption coefficient. We established the well-posedness of the model without assuming the smallness of the illumination source. This theory is closer to the situation in practical applications, where nonlinear effects are only significant with strong boundary data. For the inverse problem of reconstructing the nonlinear absorption coefficient in photoacoustic imaging, we have lifted the restrictions in~\cite{ReZh-SIAM21} by establishing an $L^1$ theory. By analyzing the spectral radius of the Peierls integral operator, we proved that the stability constants can be unified for both diffusion and transport regimes by introducing a weighted $L^1$ norm with the eigenfunction that penalizes the anisotropy near the boundary.

It is not entirely clear whether the weight function $\Psi_{\eps}$ is necessary for our attempt to unify the stability results in the transport and diffusion regimes. At a superficial level, the stability in the diffusive regime does not depend on such a weight. However, this might be due to the fact that the boundary layer effect is not taken care of by the diffusion model with the Dirichlet boundary condition in~\eqref{eq: diff}. It would be interesting to see whether a weight is also needed in the diffusive regime with Robin-type boundary conditions to account for the boundary-layer effect in the derivation of the diffusion approximation.



\section*{Acknowledgment}

We would like to thank the anonymous referees for their constructive comments that helped us improve the quality of this paper. This work is partially supported by the National Science Foundation through grants DMS-2309530 (YZ) and DMS-2309802 (KR). KR also acknowledges support from the Gordon \& Betty Moore Foundation through award GBMF12801.

\bibliographystyle{siam}
\bibliography{main}
\end{document}